\documentclass[a4paper,twoside,english]{amsart}
\usepackage{graphicx}  
\usepackage{amsmath,amssymb,amsthm, amscd}
\usepackage[pagewise]{lineno}
\usepackage{amssymb,fge}
\usepackage{tikz-cd}
\usepackage[T1]{fontenc}
\usepackage[latin9]{inputenc}
\usepackage{babel}
\usepackage{verbatim}

\makeatletter
\let\@wraptoccontribs\wraptoccontribs
\makeatother

\usepackage{lipsum}

\usepackage{color}

\usepackage[left=3.2cm,right=3.5cm,top=3cm,bottom=3cm]{geometry}
\usepackage{indentfirst}           
\usepackage{underscore}
\usepackage{enumitem}
\usepackage{relsize}
\usepackage{varwidth}
\usepackage{mathtools} 
\usepackage{hyperref}
\hypersetup{
    colorlinks=true,
    linkcolor=blue,
    filecolor=magenta,      
    urlcolor=cyan,
}
\usepackage{amssymb}
\usepackage[all]{xy}
\makeatletter
\renewcommand*\env@matrix[1][*\c@MaxMatrixCols c]{%
  \hskip -\arraycolsep
  \let\@ifnextchar\new@ifnextchar
  \array{#1}}
\makeatother

\newenvironment{Gal}{{\rm Gal\,}}{}

\def\Gal{\mathop{\rm Gal}\nolimits}

\def\ord{{\rm ord}}

\def\CVD{{\hfill\hfil{\lower 2pt\hbox{\vrule\vbox to 7pt
{\hrule width  5pt\varphifill\hrule}\varphirule}}}\par}

\newcommand{\mysetminus}{\mathbin{\fgebackslash}}

\usepackage{thmtools, thm-restate}

\newtheorem{theorem}{Theorem}[section]
\newtheorem{lemma}[theorem]{Lemma}
\newtheorem{proposition}[theorem]{Proposition}
\newtheorem{proposition-definition}[theorem]{Proposition-Definition}
\newtheorem{corollary}[theorem]{Corollary}

\theoremstyle{definition}
\newtheorem{definition}[theorem]{Definition}
\theoremstyle{remark}
\newtheorem{remark}{Remark}

\theoremstyle{theorem}

\theoremstyle{remark}

\newcommand*{\Scale}[2][4]{\scalebox{#1}{$#2$}}%

\title[]{Dynamical Diophantine Approximation Exponents in Characteristic $p$}
\author{Wade Hindes}
\begin{document}
\maketitle
\begin{abstract} Let $\phi(z)$ be a non-isotrivial rational function in one-variable with coefficients in $\overline{\mathbb{F}}_p(t)$ and assume that $\gamma\in\mathbb{P}^1(\overline{\mathbb{F}}_p(t))$ is not a post-critical point for $\phi$. Then we prove that the diophantine approximation exponent of elements of $\phi^{-m}(\gamma)$ are eventually bounded above by $\lceil d^m/2\rceil+1$. To do this, we mix diophantine techniques in characteristic $p$ with the adelic equidistribution of small points in Berkovich space. As an application, we deduce a form of Silverman's celebrated limit theorem in this setting. Namely, if we take any wandering point $a\in\mathbb{P}^1(\overline{\mathbb{F}}_p(t))$ and write $\phi^n(a)=a_n/b_n$ for some coprime polynomials $a_n,b_n\in\overline{\mathbb{F}}_p[t]$, then we prove that \vspace{.1cm} 
\[ \frac{1}{2}\leq \liminf_{n\rightarrow\infty} \frac{\deg(a_n)}{\deg(b_n)} \leq\limsup_{n\rightarrow\infty} \frac{\deg(a_n)}{\deg(b_n)}\leq2, \vspace{.1cm} \] 
whenever $0$ and $\infty$ are both not post-critical points for $\phi$. In characteristic $p$, the Thue-Siegel-Dyson-Roth theorem is false, and so our proof requires different techniques than those used by Silverman.        
\end{abstract}  
\section{Introduction}
Given an irrational number $\alpha\in\mathbb{R}$, it is a classical problem in arithmetic to study how well $\alpha$ can be approximated by rational numbers relative to their complexity. More precisely, one may try to bound the \emph{diophantine approximation exponent} of $\alpha$ given by
\begin{equation}\label{exponent-reals} 
E(\alpha):=\limsup_{h(r)\rightarrow\infty}\frac{-\log|\alpha-r|}{h(r)}.
\end{equation} 
Here $h(r)$ is the usual Weil height on $r\in\mathbb{Q}$ and $|\cdot|$ is the usual absolute value on $\mathbb{R}$. This is particularly useful in number theory when $\alpha$ is algebraic. For instance, Liouville showed that $E(\alpha)\leq d(\alpha):=[\mathbb{Q}(\alpha):\alpha]$ in this case, and many non-trivial arithmetic statements follow from improvements on the Liouville bound (e.g., the finiteness of integral points on curves). In keeping with this philosophy in the dynamical setting, Silverman \cite{SilvInt} used a generalization of Roth's Theorem (i.e., that $E(\alpha)=2$ for all algebraic irrationals) to prove the following beautiful theorem on the relative growth rate of the numerator and denominator of points in orbits. Recall that for a rational function $\phi(z)\in\mathbb{Q}(z)$, a point $P\in\mathbb{P}^1(\mathbb{Q})$ is called \emph{wandering} if $\{\phi^n(P)\}_{n\geq1}$ is infinite and called \emph{non-exceptional} if $\{\phi^{-n}(P)\}_{n\geq1}$ is infinite.   
\medskip 
\begin{theorem}\label{thm:silv}{(Silverman's Limit Theorem)} Let $\phi\in\mathbb{Q}(z)$ have degree at least $2$, let $a\in\mathbb{P}^1(\mathbb{Q})$ be a wandering point for $\phi$, and let $\phi^n(a)=a_n/b_n$ for some coprime $a_n,b_n\in\mathbb{Z}$. Then 
\[\lim_{n\rightarrow\infty} \frac{\log|a_n|}{\log|b_n|}=1,\]
whenever $0$ and $\infty$ are both non-exceptional points for $\phi$.\vspace{.25cm} 
\end{theorem}
\begin{remark} Of course, the non-exceptionality of $0$ and $\infty$ can be made explicit; it is equivalent to the condition that neither $\phi^2$ nor $1/\phi^2(1/x)$ is a polynomial. However, we prefer the first characterization since it connects more directly with our results to come.      
\end{remark} 
    

In particular, given the close relationship between number fields and global function fields, one can ask if a similar statement holds if we replace $\mathbb{Q}$ with the field $\mathbb{F}_p(t)$. However, a significant complication arises immediately: the Liouville bound (which is still valid in this setting by work of Mahler \cite{Mahler}) can be strict. Nevertheless, non-trivial improvements have been made \cite{DeMathan2,Osgood,Voloch} for ``generic" algebraic functions $\alpha\in\overline{\mathbb{F}_p(t)}$, and one can ask if it is possible to leverage these improvements in the dynamical setting. We show that the answer to this question is yes, and prove that all iterated preimages $\alpha\in\phi^{-m}(\gamma)$ of a given point $\gamma$ under a non-isotrivial map $\phi$, either have small algebraic degree (relative to $m$) or else have diophantine approximation exponent strictly smaller than the Liouville bound. Recall that a rational map $\phi(z)$ defined over a function field $K/k(t)$ is called \emph{isotrivial} if there is a change of coordinates on $\mathbb{P}^1(\bar{K})$ over which $\phi$ is defined over $\bar{k}$. Likewise, a point $\gamma\in\mathbb{P}^1(K)$ is called \emph{post-critical} for $\phi$ if its forward orbit $\{\phi^n(\gamma)\}_{n\geq1}$ contains a critical point for $\phi$; see also Remark \ref{rem:ramification}. Finally, we define the dynamical approximation exponent $E(\alpha)$ of an algebraic function $\alpha\in\overline{\mathbb{F}_p(t)}$ in an analogous way to \eqref{exponent-reals} by extending the absolute value on $\overline{\mathbb{F}}_p(t)$ given by the order vanishing at infinity; see Section \ref{sec:diophantine} for details. In particular, we prove the following non-trivial upper bound on the diophantine approximation exponent for iterated preimeges:     
\medskip       

\begin{theorem}\label{thm:eventualExp} Let $\phi\in \overline{\mathbb{F}}_p(t)(z)$ be a non-isotrivial map of degree $d\geq2$, let $v=\ord_\infty(\cdot)$ be the valuation on $\overline{\mathbb{F}}_p(t)$ given by order vanishing at infinity, and assume that $\gamma\in\mathbb{P}^1(\overline{\mathbb{F}}_p(t))$ is not a post-critical point for $\phi$. Then \[E_v(\alpha)\leq \lceil d^m/2\rceil+1\] 
for all $\alpha\in\phi^{-m}(\gamma)\mysetminus \mathbb{P}^1(\overline{\mathbb{F}}_p(t))$ and all $m\gg0$.   
\end{theorem} 

\medskip 

As an application, we combine the result above with local analysis to prove the following weakened version of Silverman's Limit Theorem in characteristic $p$: 

\medskip 

\begin{theorem}\label{thm:basic} Let $\phi\in \overline{\mathbb{F}}_p(t)(z)$ be non-isotrivial of degree at least two, let $a\in\mathbb{P}^1(\overline{\mathbb{F}}_p(t))$ be a wandering point for $\phi$, and let $\phi^n(a)=a_n/b_n$ for some coprime $a_n,b_n\in\overline{\mathbb{F}}_p[t]$. Then 
\[ \frac{1}{2}\leq \liminf_{n\rightarrow\infty} \frac{\deg(a_n)}{\deg(b_n)} \leq\limsup_{n\rightarrow\infty} \frac{\deg(a_n)}{\deg(b_n)}\leq2, \vspace{.1cm} \] 
whenever $0$ and $\infty$ are both not post-critical points for $\phi$. \vspace{.25cm}
\end{theorem}   
\begin{remark} It is worth pointing out that Theorem \ref{thm:basic} implies that the forward orbit of any wandering point contains only finitely many polynomials (corresponding to the number of $n$ such that $\deg(b_n)=0$). This finiteness statement also follows from recent work in \cite{char-p}, though the result above is much stronger. On the other hand, the assumptions in \cite{char-p} are weaker and the results there apply more generally to function fields $K/\mathbb{F}_p(t)$ and to more general notions of integral points.  \vspace{.1cm}   
\end{remark} 
\begin{remark}\label{rem:sketch} There are two reasons why we assume that the relevant points in our theorems are not post-critical, instead of the weaker condition of being not exceptional; this is useful to explain since it sheds some light on our proof strategy. First, in order to apply known equidistribution theorems for Galois orbits of points of small height (our main tool for proving that iterated preimages eventually become ``generic" enough to beat the Lioville bound) we need these iterated preimages to generate separable extensions; see Lemma \ref{lem:sep}.  
The second reason comes from the tension between the effects of ramification on local analysis and the delicate nature of diophantine approximation in characteristic $p$. Roughly speaking, to prove Theorem \ref{thm:basic} we want to show that $\lambda_v(\phi^n(a),\gamma)/h(\phi^n(a))<1$ for all $n\gg0$, where $\lambda_v(\cdot,\cdot)$ is a proximity fucntion related to the chordal metric on $\mathbb{P}^1(\mathbb{C}_v)$; see Section \ref{sec:int}. To do this, we fix a large $m$ and assume that $n>m$. Then, by a multivalued version of the inverse function theorem, 
\[\lambda_v(\phi^n(a),\gamma)\leq e_{m,\gamma}\cdot v(\phi^{n-m}(a)-\alpha_{m,n})+O(1)\vspace{.1cm}\] 
for some preimage $\alpha_{m,n}\in\phi^{-m}(\gamma)$ and some integer $e_{m,\gamma}\geq1$ depending on the ramification of $\phi^m$. On the other hand, our theorem on diophantine approximation exponents roughly says that $v(\phi^{n-m}(a)-\alpha_{m,n})\leq d^m/2\cdot h(\phi^{n-m}(a))+O(1)$. In particular, 
\[\frac{\lambda_v(\phi^n(a),\gamma)}{h(\phi^n(a))}\leq \frac{e_{m,\gamma}}{2}+o(1)\]
follows from basic properties of heights. Therefore, we cannot hope to succeed in proving $\lambda_v(\phi^n(a),\gamma)/h(\phi^n(a))<1$ with this method unless $e_{m,\gamma}=1$ for all $m$, a condition equivalent to $\gamma$ not being post-critical. In short: even when we can improve the Liouville bound, we can only improve it by a factor of $1/2$, and this improvement is swallowed up by any ramification.                    
\end{remark} 
An outline of our paper is as follows. In Section \ref{sec:equidistribution}, we prove that the algebraic functions in $\phi^{-m}(\gamma)$ either have small algebraic degree ($\leq d^m/2$) or have four Galois conjugates with non-constant cross ratio; see Theorem \ref{thm:crossratios}. The proof uses the equidistribution of points of small height in Berkovich space and an analysis of the Julia set of $\phi$ at a place of bad reduction; see also \cite{Baker-RumleyBook,Juan+Favre}. In Section \ref{sec:diophantine}, we give a brief  overview of the theory of diophantine approximation in characteristic $p$, noting in particular the importance of non-constant cross ratios to improvements of the Liouville bound; see also \cite{DeMathan,Osgood,Voloch}. This work culminates in a proof of Theorem \ref{thm:eventualExp}. Finally in Section \ref{sec:int}, we fill in the details of the sketch in Remark \ref{rem:sketch} above, combining our diophantine approximation result with local analysis to prove a version of Silverman's Limit Theorem in characteristic $p$.            

\medskip  

\noindent {\bf Acknowledgments.}  We are happy to thank Rob Benedetto, Laura DeMarco, Juan Rivera-Letelier, Tom Tucker, and Felipe Voloch for helpful conversations related to the work in this paper. We also thank MSRI for their support during the spring semester of 2022, when this project began.

\medskip

\section{Equidistribution and Cross-Ratios}\label{sec:equidistribution}
The main tool we use from dynamics to improve the Liouville bound for iterated preimages is the following strengthened version of \cite[Proposition 3.2]{char-p} for Galois orbits. 
\smallskip 
\begin{theorem}\label{thm:crossratios} Let $K/k(t)$ be a function field, let $\phi\in K(z)$ be non-isotrivial of degree $d\geq2$, and assume that $\gamma\in \mathbb{P}^1(K)$ is not a post-critical point for $\phi$. Then for all $m\gg 0$ and all $\alpha\in\phi^{-m}(\gamma)$, one of the following statements must hold: \vspace{.2cm}   
\begin{enumerate} 
\item[\textup{(1)}] The algebraic degree $d(\alpha)=[K(\alpha):K]$ of $\alpha$ is bounded by $d(\alpha)\leq d^m/2$. \vspace{.2cm}    
\item[\textup{(2)}] There exist distinct $\Gal(K^{\textup{sep}}/K)$-conjugates $\alpha_1,\alpha_2,\alpha_3,\alpha_4$ of $\alpha$ such that the cross-ratio $(\alpha_1,\alpha_2;\alpha_3,\alpha_4)$ is not in $\overline{k}$. \vspace{.1cm} 
\end{enumerate}      
\end{theorem} 
\smallskip 
To prove this result, we use tools from non-archimedean dynamics, which we now discuss. In what follows, let $(\mathbb{C}_v,|\cdot|_v)$ be an algebraically closed and complete ultrametric field, and let $\mathbb{A}^{1,\text{an}}_{\mathbb{C}_v}$ and $\mathbb{P}^{1,\text{an}}_{\mathbb{C}_v}$ be the Berkovich affine and projective lines, respectively. We recall some basic facts about these spaces; see \cite[\S6]{Rob'sBook} for more details and additional information.  

Let $\zeta=\Vert \cdot \Vert_\zeta\in\mathbb{A}^{1,\text{an}}_{\mathbb{C}_v}$ be a multiplicative seminorm on the polynomial ring $\mathbb{C}_v[T]$, and define the \emph{diameter} of $\zeta$ to be \vspace{.1cm} 
\[\text{diam}(\zeta):=\inf\{\Vert T-a\Vert_\zeta\, :\, a\in\mathbb{C}_v\}. \vspace{.1cm}\] 
In particular, if we let $\infty$ be the unique point in $\mathbb{P}^{1,\text{an}}_{\mathbb{C}_v}$ not in $\mathbb{A}^{1,\text{an}}_{\mathbb{C}_v}$ and define $\text{diam}(\infty)=+\infty$, then the function $\text{diam}:\mathbb{P}^{1,\text{an}}_{\mathbb{C}_v}\rightarrow \mathbb{R}\cup\{+\infty\}$ is upper semicontinuous and vanishes precisely on the Type I points of $\mathbb{A}^{1,\text{an}}_{\mathbb{C}_v}$. 

\begin{remark}\label{rem:diam} In fact, it is straightforward to check that if $\zeta=\zeta(a,r)$ is a Type II or III point corresponding to the closed disk $\bar{D}(a,r)$ in $\mathbb{C}_v$ for $r>0$, then $\text{diam}(\zeta)=r$. Likewise, if $\zeta$ is a Type IV point defined by a decreasing sequence of discs $\bar{D}(a_1,r_1)\supset \bar{D}(a_2,r_2)\supset\dots$ with empty intersection, then $\text{diam}(\zeta)=\lim_{n\rightarrow\infty} r_n>0$ by \cite[Lemma 2.7(b)]{Rob'sBook}.  
\end{remark} 

Now, since $\mathbb{P}^{1,\text{an}}_{\mathbb{C}_v}$ is uniquely path connected (see \cite[Theorem 6.32]{Rob'sBook}), we let $[\zeta_1,\zeta_2]\subset\mathbb{P}^{1,\text{an}}_{\mathbb{C}_v}$ be the unique path between $\zeta_1,\zeta_2\in \mathbb{P}^{1,\text{an}}_{\mathbb{C}_v}$. Then it follows from basic properties of Berkovich space that $\text{diam}: [\zeta,\infty]\rightarrow[\text{diam}(\zeta),+\infty]$ is a homeomorphism. Likewise, for $\zeta_1, \zeta_2\in \mathbb{P}^{1,\text{an}}_{\mathbb{C}_v}$, define $\zeta_1\vee\zeta_2$ to be the unique point in $\mathbb{P}^{1,\text{an}}_{\mathbb{C}_v}$ satisfying \vspace{.1cm}
\[[\zeta_1,\infty]\cap[\zeta_2,\infty]\cap[\zeta_1,\zeta_2]=\zeta_1\vee\zeta_2. \vspace{.1cm} \]
See \cite[Proposition 6.35]{Rob'sBook} for a justification. Moreover, it is clear with this definition that $\zeta_1\vee\zeta_1=\zeta_1$ and $\zeta_1\vee\zeta_2=\zeta_2\vee\zeta_1$ for all $\zeta_1,\zeta_2\in\mathbb{P}^{1,\text{an}}_{\mathbb{C}_v}$. Now (as in \cite[p.9]{Juan+Favre}) we define \vspace{.1cm}
\begin{equation}
\label{Hsia}
\sup\{\zeta_1,\zeta_2\}=\text{daim}(\zeta_1\vee\zeta_2), \vspace{.1cm} 
\end{equation} 
known as the Hsia Kernel in \cite[\S 4]{Baker-RumleyBook}. This pairing extends the metric on $\mathbb{C}_v$ and has many similar properties, several of which we collect below; see, for instance,  \cite[Proposition 4.1]{Baker-RumleyBook}. 

\begin{proposition}\label{prop:Hsia}{(Properties of the Hsia Kernel)} Let $\sup\{\cdot,\cdot\}$ be as in \eqref{Hsia}. Then: \vspace{.1cm} 
\begin{enumerate} 
\item[\textup{(a)}] $\sup\{\cdot,\cdot\}$ is nonnegative, symmetric, and continuous in each variable separately. \vspace{.2cm} 
\item[\textup{(b)}] $\sup\{\zeta_1,\zeta_2\}=|\zeta_1-\zeta_2|_v$ for all Type I points $\zeta_1,\zeta_2\in \mathbb{A}^{1,\text{an}}_{\mathbb{C}_v}$.\vspace{.2cm} 
\item[\textup{(c)}] For all $\zeta_1,\zeta_2,\zeta_3\in \mathbb{A}^{1,\text{an}}_{\mathbb{C}_v}$ we have that  
\[\sup\{\zeta_1,\zeta_3\}\leq \max\{\sup\{\zeta_1,\zeta_2\},\; \sup\{\zeta_2,\zeta_3\}  \}\]
with equality whenever $\sup\{\zeta_1,\zeta_2\}\neq \sup\{\zeta_2,\zeta_3\}$. \vspace{.2cm} 
\item[\textup{(d)}] For all fixed $\zeta\in \mathbb{A}^{1,\text{an}}_{\mathbb{C}_v}$ and $r>\textup{diam}(\zeta)$, the set 
\[\{\eta\in \mathbb{A}^{1,\text{an}}_{\mathbb{C}_v}\,:\, \sup\{\zeta,\eta\}<r\}\]
is open and connected in the Berkovich Topology. In fact, it is equal to an open Berkovich disk $D_{\textup{an}}(a,r)=\{\eta\in \mathbb{A}^{1,\text{an}}_{\mathbb{C}_v}\,:\, \Vert T-a\Vert_\eta<r\}$ for some $a\in\mathbb{C}_v$.          
\end{enumerate} 
\end{proposition}  

\begin{remark}\label{rem:Fact} A particularly useful (and easy) consequence of part (c) above is the following fact. Let $\zeta_1,\zeta_2\in\mathbb{A}^{1,\text{an}}_{\mathbb{C}_v}$ be fixed. Then for all $\eta_1,\eta_2\in\mathbb{A}^{1,\text{an}}_{\mathbb{C}_v}$ satisfying \vspace{.05cm}   
\[\max\big\{\sup\{\zeta_1,\eta_1\}\,,\, \sup\{\zeta_2,\eta_2\}\big\}< \sup\{\zeta_1,\zeta_2\}, \vspace{.05cm}\]
we have that $\sup\{\zeta_1,\zeta_2\}=\sup\{\eta_1,\eta_2\}$.      
\end{remark} 
 
From here, we may define the \emph{cross-ratio} of distinct points $\zeta_1,\zeta_2,\zeta_3,\zeta_4\in\mathbb{A}^{1,\text{an}}_{\mathbb{C}_v}$ by 
\begin{equation}\label{def:cross}
(\zeta_1,\zeta_2;\zeta_3,\zeta_4)=\frac{\sup\{\zeta_1,\zeta_4\}\cdot \sup\{\zeta_2,\zeta_3\}}{ \sup\{\zeta_1,\zeta_3\}\cdot \sup\{\zeta_2,\zeta_4\} }.
\end{equation} 
Note that by Proposition \ref{prop:Hsia} part (b), this definition of the cross-ratio on $\mathbb{A}^{1,\text{an}}_{\mathbb{C}_v}$ generalizes the usual cross-ratio on $\mathbb{C}_v$ (with absolute values added). Moreover, \eqref{def:cross} is invariant under coordinate change on $\mathbb{P}^1(\mathbb{C}_v)$. Now, in order to analyze the cross-ratio under preimages of rational functions, we discuss some tools from dynamics on Berkovich space. Let $f\in \mathbb{C}_v(z)$ be a rational map of degree at least two and let $\mathcal{J}_{f,\text{an}}$ be the (Berkovich) Julia set of $f$ in $\mathbb{P}_{\mathbb{C}_v}^{1,\text{an}}$; see \cite[\S8]{Rob'sBook}. Moreover, we say that $f$ has \emph{genuine bad reduction} over $\mathbb{C}_v$ if there is no coordinate change on $\mathbb{P}^1(\mathbb{C}_v)$ over which $f$ has good reduction (meaning the corresponding reduction map on the residue field has the same degree as the original map $f$). Next, to pass from properties of the Julia set to points in $f^{-m}(\gamma)$ for $\gamma\in\mathbb{P}^1(\mathbb{C}_v)$, we need the following notion of equidistribution. In what follows, $\mu_f$ denotes the equilibrium measure on $\mathbb{P}^{1,\text{an}}_{\mathbb{C}_v}$ attached to $f$; see \cite{Juan+Favre} for a definition that works in any characteristic. In particular, $\mu_f$ is supported precisely on the Julia set $\mathcal{J}_{f,\text{an}}$.    
\begin{definition} Let $\{S_n\}_{n\geq1}$ be a sequence of finite sets in $\mathbb{P}_{\mathbb{C}_v}^{1,\text{an}}$ and let 
\[\nu_n=\frac{1}{\#S_n}\sum_{z\in S_n}[z],\]
where $[z]$ denotes the dirac measure attached to $z\in\mathbb{P}_{\mathbb{C}_v}^{1,\text{an}}$. Then we say that the sequence $\{S_n\}_{n\geq1}$ \emph{equidistributes} with respect to $\mu_f$ if the sequence of measures $\{\nu_n\}_{n\geq1}$ converges weakly to $\mu_f$; see \cite[p.7]{prob} for a formal definition.        
\end{definition} 
\begin{remark}\label{rem:concrete} More concretely for our purposes, if $D_{\textup{an}}(a,r)\subseteq \mathbb{P}_{\mathbb{C}_v}^{1,\text{an}}$ is an open Berkovich disk, if $f$ has genuine bad reduction over $\mathbb{C}_v$, and if $\{S_n\}_{n\geq1}$ equidistributes with respect to $\mu_f$, then 
\[\mu_f(D_{\textup{an}}(a,r))=\lim_{n\rightarrow\infty}=\nu_n(D_{\textup{an}}(a,r))=\lim_{n\rightarrow\infty} \frac{\#(S_n\cap D_{\textup{an}}(a,r))}{\#S_n}\]
by the Portmanteau Theorem in probability; see, for instance, \cite[Theorem 2.1]{prob}. Here we use also that the boundary of $D_{\textup{an}}(a,r)$ is the single point $\zeta(a,r)$ and that $\mu_f$ is a non-atomic measure \cite[Theorem E]{Juan+Favre}. In particular, if $z\in\mathcal{J}_{f,\text{an}}$ and $z\in D_{\textup{an}}(a,r)$, then $\mu_f(U)>0$. Hence, for all $N$ there must exist $M$ such that $\#(S_n\cap D_{\textup{an}}(a,r))\geq N$ holds for all $n\geq M$.           
\end{remark} 
With the necessary background now in place, we are ready to prove our first result in this section. Namely, we show that if $\{S_n\}_{n\geq1}$ equidistributes with respect to $\mu_f$ and $f$ has bad reduction, then we can find distinct points in $S_n$ with cross-ratio greater than one for all large $n$; compare to \cite[Proposition 3.2]{char-p}. 
\smallskip     

\begin{proposition}\label{lem:badred2} Suppose that $f$ has genuinely bad reduction over $\mathbb{C}_v$ and that $\{S_n\}_{n\geq1}$ is a sequence of finite subsets of $\mathbb{P}_{\mathbb{C}_v}^{1,\text{an}}$ that equidistribute with respect to $\mu_f$. Then for all $n\gg0$ there are distinct points $z_{1},z_{2},z_{3},z_{4}\in S_n$ such that $(z_{1},z_{2};z_{3},z_{4})>1$.     
\end{proposition} 
\begin{proof} First, since $f$ has genuine bad reduction, the set $\mathcal{J}_{f,\text{an}}$ is uncountable and contains no isolated points; see \cite[Theorem 8.15]{Rob'sBook}. In particular, we can choose four distinct points $\zeta_1,\zeta_2,\zeta_3,\zeta_4\in \mathcal{J}_{f,\text{an}}\mysetminus\{\infty\}$ of the same type pigeon-hole principle. From here, we proceed in cases. In particular, to simplify our arguments we often change coordinates. This is justified since $\mathcal{J}_{f,\text{an}}$ is coordinate independent; see \cite[Proposition 8.2]{Rob'sBook}) \vspace{.2cm} 

\textbf{Case(1):} Suppose that $\zeta_i\vee\zeta_j\in\{\zeta_1,\zeta_2,\zeta_3,\zeta_4\}$ for all $i,j$. Then \cite[Theorem 6.32]{Rob'sBook} implies that the $\zeta_i$ are concentric points of Type II or Type III; so we may write $\zeta_i=\zeta(a,r_i)$ for some distinct $r_i>0$. In particular, after a change of coordinates if necessary, we may assume that $a=0$ and that $0<r_1<r_2<r_3<r_4$. Now, Proposition \ref{prop:Hsia} part (d) and Remark \ref{rem:concrete} together imply that for all $n\gg0$ we can choose distinct points $z_1,z_2,z_3,z_4\in S_n$ such that
\begin{equation} 
\label{concentric1} 
0<r_1<\sup\{\zeta_1,z_1\}<r_2<\sup\{\zeta_2,z_3\}<r_3<\sup\{\zeta_3,z_2\}<r_4<\sup\{\zeta_4,z_4\}. 
\end{equation} 
On the other hand, it is straightforward to check that $\sup\{\zeta_i,\zeta_j\}=\max\{r_i,r_j\}$. Hence, repeated application of Proposition \ref{prop:Hsia} part (c) yields \vspace{.1cm}   
\begin{equation}\label{concentric2}
\begin{split} 
\sup\{z_1,z_4\}&=\sup\{\zeta_4,z_4\},\;\; \sup\{z_2,z_3\}=\sup\{\zeta_3,z_2\},\\[5pt] 
\sup\{z_1,z_3\}&=\sup\{\zeta_2,z_3\},\;\; \sup\{z_2,z_4\}=\sup\{\zeta_4,z_4\}. 
\end{split} 
\end{equation}
In particular, combining \eqref{concentric1} and \eqref{concentric2} above we see that \vspace{.1cm}  
\[(z_1,z_2 ; z_3, z_4)=\frac{\sup\{\zeta_4,z_4\}\cdot \sup\{\zeta_3,z_2\}}{\sup\{\zeta_2,z_3\}\cdot \sup\{\zeta_4,z_4\}}>\frac{\sup\{\zeta_3,z_2\}}{r_3}>1,\]
and we have constructed the desired points in $S_n$.   
 \vspace{.2cm} 

\textbf{Case(2):} Suppose that $\zeta_i\vee\zeta_j\not\in\{\zeta_1,\zeta_2,\zeta_3,\zeta_4\}$ for some $i,j$. Without loss, we assume that $\zeta_1\vee \zeta_2\in(\zeta_1,\zeta_2)$. Then \cite[Proposition 6.35]{Rob'sBook} implies that $\zeta_1\vee \zeta_2$ is a point of Type II. Hence, we may change coordinates if necessary to assume that $\zeta_1\vee\zeta_2=\zeta(0,1)$. 
Therefore,  
\[\max\{\text{diam}(\zeta_1),\text{diam}(\zeta_2)\}<\sup\{\zeta_1,\zeta_2\}=1\]
by Remark \ref{rem:diam}. On the other hand, Proposition \ref{prop:Hsia} part (d) and Remark \ref{rem:concrete} then together imply that we can choose distinct points $z_1,z_2,z_3,z_4\in S_n$ such that 
\begin{equation}
\label{important} 
\max\big\{\sup\{\zeta_1,z_1\},\sup\{\zeta_1,z_3\},\sup\{\zeta_2,z_2\},\sup\{\zeta_2,z_4\}\big\}<\sup\{\zeta_1,\zeta_2\}
\end{equation}
for all $n\gg0$. However, repeated application of Remark \ref{rem:Fact} implies that
\[\sup\{z_1,z_4\}=\sup\{z_2,z_3\}=\sup\{\zeta_1,\zeta_2\}=1.\]
In particular, we see that the cross-ratio 
\begin{equation*}
\begin{split} 
(z_1,z_2:z_3,z_4)&=\sup\{z_1,z_3\}^{-1}\cdot \sup\{z_2,z_4\}^{-1}\\[5pt]
&\geq\max\{\sup\{z_1,\zeta_1\},\sup\{z_3,\zeta_1\}\}^{-1}\cdot \max\{\sup\{z_2,\zeta_2\},\sup\{z_4,\zeta_2\}\}^{-1}\\[5pt] 
&> \sup\{\zeta_1,\zeta_2\}^{-1}\cdot \sup\{\zeta_1,\zeta_2\}^{-1}=1
\end{split} 
\end{equation*}
by combining Proposition \ref{rem:Fact} part (c) and \eqref{important}. Hence, $(z_{1},z_{2};z_{3},z_{4})>1$ as desired.              
\end{proof} 

\smallskip 
Next, we pass from global information to local information using the following landmark result, which says that the images of global points of small height equidistribute in Berkovich space at every completion. The following version is stated for product formula fields (e.g., number fields and function fields of curves) and may be found in \cite[Theorem 10.24]{Baker-RumleyBook}; see also \cite{Baker-Rumley,Juan+Favre,Thuillier} for similar versions. In what follows, given a product formula field $K$ and a place $v\in\mathcal{M}_K$, we let $K_v$ denote a completion of $K$ at $v$ and let $\mathbb{C}_v$ denote a completion of $\overline{K_v}$. In particular, $\mathbb{C}_v$ is a an algebraically closed and complete ultrametric field, and we may identify $K$ with its image in $\mathbb{C}_v$. Moreover, given a rational function $\phi\in K(z)$, we let $\hat{h}_\phi: \mathbb{P}^1(K)\rightarrow\mathbb{R}$ denote the usual Call-Silverman canonical height.         
\smallskip  
\begin{theorem}\label{thm:equidistribution}{(Adelic Equidistribution of Small Points)} Let $K$ be a product formula field, let $\phi\in K(z)$ have degree at least two, and let $S_n\subseteq K^{\textup{sep}}$ be a finite $\Gal(K^{\textup{sep}}/K)$-stable set for each $n\geq1$. Moreover, assume that $|S_n|\rightarrow\infty$ as $n\rightarrow\infty$ and that
\[\lim_{n\rightarrow\infty}\frac{1}{\# S_n}\sum_{\alpha\in S_n}\hat{h}_\phi(\alpha)=0.\]
Then $\{S_n\}$ equidistributes with respect to $\mu_{\phi,v}$ for every place $v\in \mathcal{M}_K$, where $\mu_{\phi,v}$ is the equilibrium measure associated to the induced map of $\phi$ on $\mathbb{P}_{\mathbb{C}_v}^{1,\text{an}}$.     
\end{theorem} 
In particular, since we wish to apply Theorem \ref{thm:equidistribution} to certain subsets $S_m\subseteq\phi^{-m}(\gamma)$, we must be sure that the iterated preimages of $\gamma$ are separable over $K$. This is guaranteed in our case by the following elementary result.   
\begin{lemma}\label{lem:sep} Let $K$ be a field, let $\phi\in K(z)$ be non-constant map, and assume that $\gamma\in\mathbb{P}^1(K)$ is not a post-critical point for $\phi$. Then $\alpha\in\mathbb{P}^1(K^{\textup{sep}})$ for all $\alpha\in\phi^{-m}(\gamma)$ and all $m\geq1$.   
\end{lemma}
\begin{proof} 
After a change of variables, we may assume that $\gamma$ is not the point at infinity. Assume for a contradiction that $\alpha\in\phi^{-m}(\gamma)$ for some $m\geq1$ and that $\alpha\not\in\mathbb{P}^1(K^{\textup{sep}})$. Then the minimum polynomial $m(z)\in K[z]$ of $\alpha$ has a repeated root; say $m(z)=(z-c)^r\cdot h(z)$ for some $r\geq2$, some $c\in \overline{K}$, and some $h\in \overline{K}[z]$. Now write $\phi^m=f_m(z)/g_m(z)$ for some coprime $f_m,g_m\in K[z]$. Moreover, note that we can assume that $g(\alpha)\neq0$. Then, since $\phi^m(\alpha)=\gamma$, we have that $\alpha$ is a root of the polynomial $F_m:=f_m(z)-\gamma\cdot g_m(z)\in K[z]$. Hence, $m(z)$ must divide $F_m$, and so $F_m(z)=(x-c)^e\cdot H(z)$ for some $e\geq r\geq2$ and some $H\in \overline{K}[z]$ with $H(c)\neq0$. In particular, we note that $F_m(c)=0$ and that $g_m(c)\neq0$; otherwise, $f_m$ and $g_m$ have a common factor $(x-c)$ in $\overline{K}[z]$ and so have a common factor in $K[z]$, a contradiction. In particular, we deduce that $\phi^m(c)=\gamma$ and that 
\[\phi^m(z)-\gamma=(z-c)^e\cdot \frac{H(z)}{g_m(z)}=(z-c)^e\cdot R(z)\]
for some $R(z)\in \overline{K}(z)$ with $R(c)\neq0$. Therefore, $e_{\phi^m}(c)=e\geq2$ and 
\[e_\phi(c)\cdot e_\phi(\phi(c))\dots e_{\phi}(\phi^{m-1}(c))=e_{\phi^m}(c)=e\geq2\]
by the multiplicativity of ramification indices; see \cite[Proposition 2.6(c)]{SilvElliptic}. In particular, $e_{\phi}(\phi^{i}(c))\geq2$ for some superscript $0\leq i<m$, and thus $\phi^{i}(c)$ is a critical point of $\phi$; see Remark \ref{rem:ramification}. On the other hand, $\phi^{m-i}(\phi^i(c))=\phi^m(c)=\gamma$, which implies that $\gamma$ is a post-critical point for $\phi$ (since $m-i\geq1$), a contradiction. Therefore, the minimum polynomial of $\alpha$ over $K$ has no repeated roots in $\overline{K}$; thus $\alpha\in\mathbb{P}^1(K^{\text{sep}})$ as claimed.   
\end{proof}  
\smallskip 
Finally, before proving Theorem \ref{thm:crossratios}, we need the following result guaranteeing a place of genuine bad reduction for all non-isotrivial maps.  
\smallskip 
\begin{theorem}\label{thm:Baker}{(\cite[Theorem 1.9]{Baker})} Let $K/k(t)$ be a function field and let $\phi\in K(z)$ have degree at least two. Then $\phi$ is non-isotrivial if and only if there is some place $v\in \mathcal{M}_K$ at which $\phi$ has genuine bad reduction. 
\end{theorem}

\smallskip    
We now have all of the dynamical tools in place to prove our main result in this section. 
\begin{proof}[(Proof of Theorem \ref{thm:crossratios})] Let $K/k(t)$ be a function field, let $\phi\in K(z)$ be a non-isotrivial map, and assume that $\gamma\in\mathbb{P}^1(K)$ is not a post-critical point for $\phi$. Then $\alpha\in\mathbb{P}^1(K^{\text{sep}})$ for all $\alpha\in\phi^{-m}(\gamma)$ and all $m\geq1$ by Lemma \ref{lem:sep}.  
Now define the set  
\[S_m:=\{\alpha\in\phi^{-m}(\gamma)\,:\, d(\alpha)>d^m/2\}.\]
Note that if $S_m=\varnothing$ for all $m\gg0$, then there is nothing to prove. Therefore, we may assume that the set $I:=\{m\,:\,S_m\neq\varnothing\}$ is infinite. We will show that the hypothesis of Theorem \ref{thm:equidistribution} hold for the sequence $\{S_m\}_{m\in I}$. Clearly $S_m$ is $\Gal(K^{\text{sep}}/K)$-stable for all $m\in I$, as the algebraic degree of a separable element is invariant under the action of Galois. Likewise, by choosing some fixed $\alpha_0\in S_m$ for $m\in I$, we see that
\[|S_m|\geq\#\{\sigma(\alpha):\, \sigma\in\Gal(K^{\text{sep}}/K)\}=d(\alpha)>d^m/2\qquad \text{for $m\in I$,}\]
since Galois acts transitively on the roots of separable minimal polynomials. In particular, $|S_m|\rightarrow\infty$ as $m$ grows. Finally,   
\[\hat{h}_\phi(S_m):=\frac{1}{\# S_m}\sum_{\alpha\in S_m}\hat{h}_\phi(\alpha)=\frac{1}{\# S_m}\sum_{\alpha\in S_m}\frac{\hat{h}_\phi(\gamma)}{d^m}=\frac{\hat{h}_\phi(\gamma)}{d^m}\]
for all $m\in I$ by the functoriality of canonical heights. Thus, $\hat{h}_\phi(S_m)\rightarrow 0$ as $m\in I$ tends to infinity. Hence, Theorem \ref{thm:equidistribution} applies and the sequence $\{S_m\}_{m\in I}$ equidistributes with respect to the equilibrium measure $\mu_{\phi,v}$ attached to the induced map of $\phi$ on $\mathbb{P}_v^{1,\text{an}}$  for every place $v\in \mathcal{M}_K$. In particular, equidistribution holds for some place $v$ of genuine bad reduction for $\phi$ by Theorem \ref{thm:Baker} (since $\phi$ is non-isotrivial). However, in this case Proposition  \ref{lem:badred2} implies that for all sufficiently large $m\in I$ there exist $\alpha_1,\alpha_2,\alpha_3,\alpha_4\in S_m$ with non-constant cross-ratio (a cross-ratio in $\overline{k}$ necessarily has absolute value equal to one at all places). On the other hand, note that $\Gal(K^{\text{sep}}/K)$ acts transitively on $S_m$ for $m\in I$: otherwise $S_m$ is a union of at least $2$ disjoint orbits and thus $\# S_m>d^m/2+d^m/2=d^m$, contradicting the fact that $S_m\subseteq \phi^{-m}(\gamma)$. In particular, for all sufficiently large $m\in I$ and any $\alpha\in S_m$, there are $\Gal(K^{\text{sep}}/K)$-conjugates of $\alpha$ with non-constant cross ratio as claimed.                    
\end{proof} 

\section{Diophantine Approximation}\label{sec:diophantine}
We now give an overview of the theory of diophantine approximation in characteristic $p$ with particular emphasis on the importance of non-constant cross-ratios. First some notation. Throughout this section let $K=\overline{\mathbb{F}}_p(t)$ and let $v$ be a valuation on $K$; for simplicity, we work over an algebraically closed constant field, though this is not strictly speaking necessary. For such $K$, define the \emph{height} of $r\in K$ to be $h(r)=\max\{\deg(a),\deg(b)\}$, where $r=a/b$ for some coprime $a,b\in\overline{\mathbb{F}}_p[t]$. Analogous to the case of approximating algebraic irrationals, the goal in this setting is to understand how well a given algebraic function can be approximated by elements of $K$ with respect to $v$, relative to the height of the approximation. With this in mind, we define the \emph{diophantine approximations exponent} of $\alpha\in\overline{K}\mysetminus K$ with respect to the place $v$ to be 
\begin{equation}\label{def:diophantineExp} 
E_v(\alpha):=\max_{\substack{w\in \mathcal{M}_{K(\alpha)}\\ w|v}}\Bigg\{\limsup_{\substack{h(r)\rightarrow\infty\\ r\in K}}\frac{w(\alpha-r)}{h(r)}\Bigg\}. 
\end{equation} 
Here $\mathcal{M}_{K(\alpha)}$ is a complete set of inequivalent valuations on $K(\alpha)$ and $w\big|v$ indicates that $w$ restricts to $v$ on the subset $K\subseteq K(\alpha)$. 
Of particular interest for us are the exponents $E_v(\alpha)$ when $v$ is the valuation given by the order vanishing at infinity, $v(a/b)=\ord_\infty(a/b)=\deg(a)-\deg(b)$. In this case, the first important result in this setting is due to Mahler \cite{Mahler}:  
\smallskip

\begin{theorem}\label{thm:Liouville}{(The Liouville Bound)} Let $K=\overline{\mathbb{F}}_p(t)$ and let $v=\ord_\infty(\cdot)$. If $\alpha\in \overline{K}\mysetminus K$ has degree $d(\alpha)=[K(\alpha):K]\geq2$ and $w\in \mathcal{M}_{K(\alpha)}$ is any extension of $v$, then there is a constant $C=C(\alpha,w)$ such that 
\[w(\alpha-r)\leq d(\alpha)\cdot h(r)+C \;\;\;\;\; \text{for all $r\in K$}. \vspace{.1cm} \]
In particular, $E_v(\alpha)\leq d(\alpha)$ for all $\alpha\in \overline{K}\mysetminus K$.     
\end{theorem}
\smallskip 

However important to the general theory, the Liouville bound must be improved to deduce many arithmetic applications (as over number fields). Unfortunately, unlike the number field case, the Liouville bound is in general strict in characteristic $p$; for instance, $E_v(\alpha)=p=d(\alpha)$ for $\alpha$ satisfying $\alpha^p-\alpha-t^{-1}=0$. Nevertheless, an improvement of the Liouville bound is still possible for certain algebraic functions, including those for which \emph{some} conjugates have non-constant cross-ratio:             
\smallskip  
\begin{theorem}[Osgood-Voloch]\label{thm:Osgood-Voloch} Let $K=\overline{\mathbb{F}}_p(t)$ and let $v=\ord_\infty(\cdot)$. If $\alpha\in\overline{K}\mysetminus K$ has degree $d(\alpha)=[K(\alpha):K]\geq4$ and $w\in \mathcal{M}_{K(\alpha)}$ is any extension of $v$, then there is a constant $C=C(\alpha,w)$ such that 
\begin{equation}\label{thm:osgoodbd}
w(\alpha-r)\leq \big(\lceil d(\alpha)/2\rceil+1\big)\cdot h(r)+C \;\;\;\;\; \text{for all $r\in K$},
\end{equation} 
unless $\alpha\in K^{\textup{sep}}$ and the cross-ratio of any four distinct conjugates of $\alpha$ is in $\overline{\mathbb{F}}_p$. In particular, $E_v(\alpha)\leq \lceil d(\alpha)/2\rceil+1$ for such $\alpha \in \overline{K}\mysetminus K$.  \vspace{.1cm}         
\end{theorem} 
\begin{remark} In fact, a stronger result was obtained by Lasjaunias and de Mathan in \cite{DeMathan}; they prove the upper bound in \eqref{thm:osgoodbd} for all $\alpha\in \overline{K}\mysetminus K$ not satisfying a \emph{Frobenius equation}, i.e, $\alpha=(a\alpha^{p^n}+b)/(c\alpha^{p^n}+d)$ for some $a,b,c,d\in K$ and some $n\geq1$. Nevertheless, we use the formulation above with cross-ratios, since it more easily relates to equidistribtuion theorems in dynamics; for instance, see Proposition \ref{lem:badred2} above.          
\end{remark} 
\begin{remark}
Theorem \ref{thm:Osgood-Voloch} is essentially a restatement of \cite[Lemma 3.2]{DeMathan} or \cite[Theorem 3]{Voloch}, up to a slight change in setup: these results are stated for algebraic elements of a fixed completion $K_v$ of $K$. However, roughly speaking if an element $\alpha\in\overline{K}$ cannot be embedded in a completion at $v$, then it cannot be arbitrarily close to an element of $K$ for any extension $w|v$ (i.e., $w(\alpha-r)\leq C(\alpha)$ for all $r\in K$), and so the bound in \eqref{thm:osgoodbd} holds trivially in this case. The proof of Theorem \ref{thm:Osgood-Voloch} combines Osgood's original result \cite[Theorem III]{Osgood} (which is stated in the form we use above for all $d(\alpha)\geq2$) with an argument due to Voloch using cross-ratios. 
\end{remark} 
\begin{proof}[(Proof Sketch)] If $\alpha\in\overline{K}\mysetminus K$ is such that the bound in \eqref{thm:osgoodbd} fails, then \cite[Theorem III]{Osgood} implies that $\alpha\in K^{\text{sep}}$ and that $\alpha$ must satisfy a \emph{Ricatti equation} over $K$:  
\[\frac{d\alpha}{dt}=a\alpha^2+b\alpha+c\]
for some $a,b,c\in K$; here the derivative of $\alpha$ is given by implicit differentiation of it's minimal polynomial, which makes sense by separability. In fact, Osgood's proof gives additional information: there are infinitely many $r_n\in K$ all satisfying the same Ricatti equation as $\alpha$ above. In particular, there are three such solutions $r_1,r_2,r_3\in K$. However, as in the classical theory, the derivative of the cross ratio of any four solutions of the same Ricatti equation must be zero. In particular, the derivative of $(\alpha,r_1;r_2,r_3)\in K^{\text{sep}}$ must be zero. Hence, $(\alpha,r_1;r_2,r_3)$ is a $p$th power (see, for instance \cite[p.111]{Osgood}), and thus   
\[(\alpha,r_1;r_2,r_3)=\alpha_1^p\]
for some $\alpha_1\in K(\alpha)$. On the other hand, the function $z\rightarrow(z,r_1;r_2,r_3)$ is a linear fractional transformation $L_1(z)\in K(z)$, so that
$L_1(\alpha)=\alpha_1^p$ for some $\alpha_1\in K(\alpha)$ and some $L_1(z)\in K(z)$ of degree $1$. Moreover $d(\alpha)=d(\alpha_1)$, since 
\[K(\alpha_1)\subseteq K(\alpha)\subseteq K(\alpha_1^p)\subseteq K(\alpha_1).\]
In fact, the bound in \eqref{thm:osgoodbd} must also fail for $\alpha_1$; see \cite[p.3]{DeMathan} or \cite[p.5]{Voloch}. In particular, we may iterate the argument above to construct a sequence $\alpha_n\in K(\alpha)$ and $L_n\in K(z)$ of degree $1$ such that 
\[L_n(\alpha)=\alpha_n^{\,p^n}.\] 
Now assume that $d(\alpha)\geq 4$ and fix any distinct conjugates $\sigma_1(\alpha),\sigma_2(\alpha),\sigma_3(\alpha),\sigma_4(\alpha)$ of $\alpha$ for some $\sigma_i\in\Gal(K^{\text{sep}}/K)$. Then, since cross ratios are invariant under linear fractional transformations and the inverse $L_n^{-1}$ of $L_n$ has coefficients in $K$, we have that \vspace{.15cm}  
\begin{equation*} 
\begin{split} 
\big(\sigma_1(\alpha),\sigma_2(\alpha)\,;\sigma_3(\alpha),\sigma_4(\alpha)\big)&=\big(\,\sigma_1(L_n^{-1}(\alpha_n^{p_n})),\, \sigma_2(L_n^{-1}(\alpha_n^{p_n}))\,;\, \sigma_3(L_n^{-1}(\alpha_n^{p_n})),\, \sigma_4(L_n^{-1}(\alpha_n^{p_n}))\,\big)\\[5pt] 
&= \big(\sigma_1(\alpha_n),\sigma_2(\alpha_n)\,;\sigma_3(\alpha_n),\sigma_4(\alpha_n)\big)^{p^n} \\[2pt]   
\end{split} 
\end{equation*}
for all $n\geq1$. In particular, if we let $L=K(\sigma_1(\alpha),\sigma_2(\alpha),\sigma_3(\alpha),\sigma_4(\alpha))$, then the calculation above and the fact that $\alpha_n\in K(\alpha)$ for all $n$ together imply that \vspace{.1cm}  
\[ \big(\sigma_1(\alpha),\sigma_2(\alpha)\,;\sigma_3(\alpha),\sigma_4(\alpha)\big)\in \bigcap_{n\geq1}L^{p^n}.\] 
However, $\bigcap_{n\geq1}L^{p^n}\subseteq \overline{\mathbb{F}}_q$; to see this, note that if $u$ is any discrete valuation on $L$ and $\kappa\in \bigcap_{n\geq1}L^{p^n}$, then $u(\kappa)$ is divisible by $p^n$ for all $n\geq1$. Hence, $u(\kappa)=0$ for all $u$, and thus $\kappa$ must be a constant; see, for instance, \cite[Corollary 1.1.20]{FFieldsandCodes}.                  
\end{proof} 

\smallskip 

We now have the tools in place to give a quick proof of our bound on the diophantine approximation exponent of iterated preimages from the Introduction.

\smallskip 

\begin{proof}[(Proof of Theorem \ref{thm:eventualExp})] Let $\phi\in\overline{\mathbb{F}}_p(t)(z)$ be a non-isotrivial map of degree $\deg_z(\phi)\geq2$ and let $\gamma\in\mathbb{P}^1(\overline{\mathbb{F}}_p(t))$ be a non-postcritical point for $\phi$. Now let $K=\overline{\mathbb{F}}_p(t)$, let $v$ be the valuation on $K$ corresponding to the order vanish at infinity, and choose $m\gg0$ as in Theorem \ref{thm:crossratios}. Then for all $\alpha\in \phi^{-m}(\gamma)\mysetminus \mathbb{P}^1(K)$, either $d(\alpha)=[K(\alpha):K]\leq d^m/2$ or there exist some distinct $ \Gal(K^{\text{sep}}/K)$ conjugates of $\alpha$ with cross-ratio not in $\overline{\mathbb{F}}_p$. Then, $E_v(\alpha)\leq d^m/2$ in the first case by the Liouville bound; see Theorem \ref{thm:Liouville}. On the other hand, $E_v(\alpha)\leq \lceil d(\alpha)/2\rceil+1\leq \lceil d^m/2\rceil+1$ in the second case by Theorem \ref{thm:osgoodbd}; here we use also that $d(\alpha)\leq d^m$, since $\phi^m\in K(z)$ has degree $d^m$ and $\phi^m(\alpha)=\gamma\in K$. In particular, $E_v(\alpha)\leq \lceil d^m/2\rceil+1$ holds for all $\alpha\in\phi^{-m}(\gamma)$ as claimed.  
\end{proof} 
\section{Integrality estimates for forward orbits}\label{sec:int}
To prove our version of Silverman's Limit Theorem in characteristic $p$ from the Introduction, we combine the results in Section \ref{sec:equidistribution} and Section \ref{sec:diophantine} with some additional local analysis. With this in mind, we fix some notation. Let $K_v$ be a complete field and let $|\cdot|_v$ be the unique extension of the absolute value on $K_v$ to $\overline{K_v}$. Then we may define the \emph{chordal metric} on $\mathbb{P}^1(\overline{K_v})$ by \vspace{.1cm}  
\[\rho_v\big([x_0,x_1],[y_0,y_1]\big):=\frac{|x_0y_1-y_0x_1|_v}{\max\{|x_0|_v,|x_1|_v\}\,\max\{|y_0|_v,|y_1|_v\}} \vspace{.1cm}\]
for $[x_0,x_1],[y_0,y_1]\in\mathbb{P}^1(\overline{K_v})$. Likewise, we define an associated \emph{proximity function} given by \[\lambda_v(P,Q)=-\log\rho_v(P,Q)\] 
for distinct $P,Q\in \mathbb{P}^1(\overline{K_v})$. In particular, given a map $\phi\in K_v(z)$ and points $a,\gamma\in \mathbb{P}^1(\overline{K_v})$, then it is often useful to study the growth rate of the quantity $\lambda(\phi^n(a),\gamma)$ as $n$ grows; for instance, if $K_v$ is the completion of a product formula field $K$ and the points $a$ and $\gamma$ are $K$-rational, then $\lambda(\phi^n(a),\gamma)$ provides a way to measure the integrality of $\phi^n(a)$ relative to $\gamma$. In this vein, we have the following result in the case when $K= \overline{\mathbb{F}}_p(t)$, when $v=\ord_\infty(\cdot)$, and when $K_v$ is a completion of $K$ with respect to the absolute value $|\cdot|_v=p^{-v(\cdot)}$ on $K$.              

\medskip 
\begin{theorem}\label{thm:limit} Let $\phi\in \overline{\mathbb{F}}_p(t)(z)$ be non-isotrivial of degree at least two and let $v=\ord_\infty(\cdot)$. Moreover, assume that $\gamma\in \mathbb{P}^1(\overline{\mathbb{F}}_p(t))$ is not a post-critical point for $\phi$. Then   
\[\limsup_{n\rightarrow\infty}\frac{\lambda_v\big(\phi^n(a),\gamma\big)}{h\big(\phi^n(a)\big)}\leq \frac{1}{2}.\]
for all wandering $a\in \mathbb{P}^1(\overline{\mathbb{F}}_p(t))$.  
\end{theorem} 
\medskip 
To prove this result, we combine ideas in \cite{me:JNT} and \cite{SilvInt}. In particular as a first step, we use the lemma below to control the proximity function under preimages of rational functions; this result may be viewed as a sort of multivalued inverse function theorem. In what follows, given a rational map $\varphi(z)\in K_v(z)$ and a point $\alpha\in\mathbb{P}^1(\overline{K}_v)$, we let $e_\varphi(\alpha)$ denote the ramification index of $\varphi$ at $\alpha$; see Remark \ref{rem:ramification} below. 
\medskip
\begin{lemma}\label{lem:inv} Let $(K_v,|\cdot|_v)$ be a complete field, let $\varphi(z)\in K_v(z)$ be a map of degree at least two, and let $\gamma\in\mathbb{P}^1(\overline{K}_v)$. Moreover, define   
\[\mathbf{e}_\varphi(\gamma):=\max_{\alpha\in\varphi^{-1}(\gamma)} e_\varphi(\alpha).\]
Then there is a positive constant $C=C(v,\varphi,\gamma)$ such that for all $P\in\mathbb{P}^1(\overline{K}_v)$ with $\varphi(P)\neq\gamma$ there exists a preimage $\alpha\in\varphi^{-1}(\gamma)$ (depending on $P$) satisfying \vspace{.1cm}   
\[\lambda_{v}\big(\varphi(P),\gamma\big)\leq \mathbf{e}_\varphi(\gamma)\cdot\lambda_{v}(P,\alpha)+C. \vspace{.1cm} \] 
\end{lemma}
\begin{proof} This is done in greater generality for maps $\varphi: V\rightarrow W$ of varieties of dimension $1$ and \'{e}tale maps in \cite[Theorem 5.2]{InvFunThm}. On the other hand, a simpler version exists in \cite{SilvInt} for maps of $\mathbb{P}^1$, though this result is stated in characteristic zero.   
\end{proof}
\begin{remark}\label{rem:ramification} Let $L$ be any algebraically closed field, let $\varphi\in L(z)$, and let $\alpha\in\mathbb{P}^1(L)$. Then there
 are degree one rational functions $\sigma, \theta \in L(z)$ such that
 $\theta(0)=\alpha$ and $\sigma \circ \varphi \circ \theta (0) = 0$. Thus, we may write
 $\sigma \circ \varphi \circ \theta (z) = z^e g(z)$ for some rational function $g$
 such that $g(0) \not= 0$.  We call $e\geq1$ the \emph{ramification index}
 of $\varphi$ at $\alpha$ and denote it by
 $e_{\varphi}(\alpha)$. In particular, ramification indices are multiplicative: $e_{\psi\circ\phi}(\alpha)=e_\phi(\alpha)\cdot e_{\psi}(\phi(\alpha))$. Moreover, we say that $\alpha$ is a
 \emph{critical point} of $\varphi$ if
 $e_{\varphi}(\alpha)> 1$. In particular, it follows that if $\gamma$ is not in the forward orbit of a critical point (i.e., $\gamma$ is not \emph{post-critical}), then $e_{\phi^m}(\alpha)=1$ for all $\alpha\in\phi^{-m}(\gamma)$ and all $m\geq1$.        
\end{remark}   
\medskip
Likewise, we need an elementary lemma that compares $\lambda_v(x,y)$ to the usual distance $|x-y|_v$ for certain $x,y\in\overline{K}_v$; see \cite[Lemma 3.53]{SilvDyn} for a proof of the following.
\medskip  
\begin{lemma}{\label{lem:comp}} Let $x,y\in\overline{K}_v$ be distinct and suppose that $\lambda_{v}(x,\infty)+\log(2)\leq\lambda_{v}(x,y)$. Then \vspace{.1cm} 
\[\lambda_{v}(x,y)\leq -\log|x-y|_v+2\lambda_{v}(x,\infty)+\log2. \vspace{.1cm}\] 
\end{lemma}
\begin{proof} See also \cite[Lemma 3]{Hsia-Silv}. This result is stated over number fields. However, the proof uses only basic properties of non-archimedean absolute values and so works more generally. 
\end{proof}  
\medskip 
We now have all of the tools in place to prove Theorem \ref{thm:limit}, which compares the local proximity of $\phi^n(a)$ to $\gamma$ with the height of $\phi^n(a)$.  
\begin{proof}[(Proof of Theorem \ref{thm:limit})] Let $K=\overline{\mathbb{F}}_p(t)$, let $v=\ord_\infty(\cdot)$ be the valuation corresponding to the order vanishing at infinity, and let $a\in \mathbb{P}^1(K)$ be a wandering point for a non-isotrivial map $\phi(z)\in K(z)$. Now fix a completion $K_v$ with respect to the absolute value $|\cdot|_v$ on $K$, and note that we may embed $\overline{K}\subseteq \overline{K_v}$ (e.g., by first embedding $K\subseteq K_v$). In particular if $\alpha\in\overline{K}$, then it is straightforward to check that $\bar{v}(\alpha)=-\log|\alpha|_v$ gives a valuation on $\overline{K}$ extending $v$.        

Now fix $m\gg0$ as in Theorem \ref{thm:crossratios} and apply Lemma \ref{lem:inv} to the map $\varphi=\phi^m\in K(z)\subseteq K_v(z)$. In particular, since $\gamma$ is not a post critical point for $\phi$ and ramification indices are multiplicative, we have that $\textbf{e}_\gamma(\phi^m)=1$; see Remark \ref{rem:ramification}. 
Hence, for all $n>m$ there is $\alpha_{m,n}\in\phi^{-m}(\gamma)\subseteq \mathbb{P}^1(K^{\text{sep}})$ satisfying \vspace{.05cm} 
\begin{equation}\label{eq:limit1}
\lambda_v(\phi^n(a),\gamma)=\lambda_v(\phi^m(\phi^{n-m}(a)), \gamma)\leq \lambda_v(\phi^{n-m}(a),\alpha_{n,m})+C_1(v,\phi,m) \vspace{.05cm} 
\end{equation}       
by Lemma \ref{lem:inv}. From here, we proceed in cases.   
\\[5pt] 
\textbf{Case (1):} Suppose that $\alpha_{m,n}=\infty=[1,0]\in\mathbb{P}^1(K)\mysetminus K$ and write $\phi^n(a)=a_n(t)/b_n(t)$ for some coprime polynomials $a_n(t),b_n(t)\in \overline{\mathbb{F}}_p[t]$. Then \eqref{eq:limit1} implies that  \vspace{.2cm}  
\begin{equation}{\label{eq:Case2}} 
\begin{split} 
\lambda_v(\phi^n(a),\gamma)&\leq \lambda_v(\phi^{n-m}(a),\infty)+ C_1(v,\phi,m)\\[8pt]
&=-\log\bigg(\frac{|b_{n-m}|_v}{\max\big\{|a_{n-m}|_v,|b_{n-m}|_v\big\}}\bigg)+C_1(v,\phi,m)\\[8pt]
&=\max\{\deg(a_{n-m}),\deg(b_{n-m})\}-\deg(b_{n-m})+C_1(v,\phi,m).\\[5pt] 
&\leq h(\phi^{n-m}(a))+ C_1(v,\phi,m)
\end{split} 
\end{equation}
for all $n>m$. 
\\[8pt] 
\noindent\textbf{Case(2):} Suppose that $\alpha_{m,n}\in \overline{K}$ and that the bound 
\begin{equation}\label{eq:limit2}
\lambda_v(\phi^{n-m}(a),\alpha_{n,m})\leq \lambda_v(\alpha_{m,n},\infty)+\log2
\end{equation}   
holds. Then \eqref{eq:limit1} and \eqref{eq:limit2} together imply that  \vspace{.05cm} 
\begin{equation} \label{eq:Case1} 
\begin{split} 
\lambda_v(\phi^n(a),\gamma)&\leq\lambda_v(\phi^{n-m}(a),\alpha_{n,m})+C_1(v,\phi,m) \\[5pt] 
&\leq\max_{\alpha\in\phi^{-m}(\gamma)}\lambda_v(\alpha,\infty)+\log2+C_1(v,\phi,m)\\[5pt] 
&\leq C_2(v,\phi,m) 
\end{split} 
\end{equation} 
for some constant $C_2(v,\phi,m)$ and all $n>m$.
\\[8pt]
\textbf{Case (3):} Suppose that $\alpha_{m,n}\in \overline{K}\mysetminus K$, that $d(\alpha_{m,n})=[K(\alpha):K]> d^m/2$, and that the bound in \eqref{eq:limit2} is false. In particular, \vspace{.1cm} 
\[ \lambda_v(\alpha_{m,n},\infty)+\log2\leq\lambda_v(\phi^{n-m}(a),\alpha_{n,m})=\lambda_v(\alpha_{n,m}, \phi^{n-m}(a)). \vspace{.05cm} \]
Then Lemma \ref{lem:comp} applied to $x=\alpha_{m,n}$ and $y=\phi^{n-m}(a)$ implies that \vspace{.1cm}  
\begin{equation}\label{eq:limit4}
\begin{split} 
\lambda_v(\phi^{n-m}(a),\alpha_{n,m})&\leq -\log|\alpha_{m,n}-\phi^{n-m}(a)|+2\lambda_v(\alpha_{m,n},\infty)+\log2 \\[5pt]
&=\bar{v}(\alpha_{m,n}-\phi^{n-m}(a))+\max_{\alpha\in\phi^{-m}(\gamma)}2 \lambda_v(\alpha,\infty)+\log2 \\[5pt] 
&=\bar{v}(\alpha_{m,n}-\phi^{n-m}(a))+C_3(v,\phi,m). 
\end{split} 
\end{equation}            
Hence, \eqref{eq:limit1} and \eqref{eq:limit4} together imply that \vspace{.05cm} 
\begin{equation}\label{eq:limit5}
\lambda_v(\phi^n(a),\gamma)\leq \bar{v}(\alpha_{m,n}-\phi^{n-m}(a))+C_4(v,\phi,m)
\end{equation} 
for some constant $C_4(v,\phi,m)$ and all $n>m$. On the other hand, Theorem \ref{thm:crossratios} and Theorem \ref{thm:Osgood-Voloch} (applied to the valuation $w\in M_{K(\alpha_{m,n})}$ given by the restriction of $\bar{v}$ to $K(\alpha_{m,n})\subseteq\overline{K}$) together imply that  
\begin{equation}\label{eq:limit6}
\bar{v}(\alpha_{m,n}-\phi^{n-m}(a))\leq \big(\lceil d(\alpha_{m,n})/2\rceil+1\big)\cdot h(\phi^{n-m}(a))+C(\alpha_{m,n})  
\end{equation}
by construction of $m$. Therefore, \eqref{eq:limit5} and \eqref{eq:limit6} together imply that \vspace{.15cm}
\begin{equation}\label{eq:Case3}
\lambda_v(\phi^n(a),\gamma)\leq \big( \lceil d^m/2\rceil+1\big) \cdot h(\phi^{n-m}(a))+ C_5(v,\phi,m)\vspace{.15cm}
\end{equation} 
for some $C_5(v,\phi,m)$ and all $n>m$. Here we use that $d(\alpha_{m,n})\leq d^m$ since $\alpha_{m,n}\in\phi^{-m}(\gamma)$ and $\phi$ has degree $d$.
\\[8pt]
\textbf{Case (4):} Suppose that $\alpha_{m,n}\in \overline{K}\mysetminus K$, that $d(\alpha_{m,n})=[K(\alpha):K]\leq d^m/2$, and that the bound in \eqref{eq:limit2} is false. Then, repeating the argument at the beginning of Case (3), we see that 
\[\lambda_v(\phi^n(a),\gamma)\leq\bar{v}(\alpha_{m,n}-\phi^{n-m}(a))+C_4(v,\phi,m).\]
On the other hand, the Liouville bound in Theorem \ref{thm:Liouville} implies that 
\[\bar{v}(\alpha_{m,n}-\phi^{n-m}(a))\leq d(\alpha_{m,n})h(\phi^{n-m}(a))+ C_6(v,\phi,m)\] for some constant $C_6(v,\phi,m)$. In particular,  
\begin{equation}\label{eq:Case4}
\lambda_v(\phi^n(a),\gamma)\leq \frac{d^m}{2} \cdot h(\phi^{n-m}(a))+ C_7(v,\phi,m)\vspace{.15cm}
\end{equation} 
for all $n>m$ and some constant $C_7(v,\phi,m)$.\\[5pt] 
\textbf{Case(5):} Finally, suppose that $\alpha_{m,n}\in K=\overline{\mathbb{F}}_p(t)$ and that the bound in \eqref{eq:limit2} is false. In particular, repeating the argument at the beginning of Case (3), we see that 
\[\lambda_v(\phi^n(a),\gamma)\leq v(\alpha_{m,n}-\phi^{n-m}(a))+C_4(v,\phi,m).\]
On the other hand, it is straightforward to check that $v(A-B)\leq \max\{h(A),h(B)\}$ for all distinct $A,B\in\overline{\mathbb{F}}_p(t)$. In particular, since $h(\phi^n(a))\rightarrow\infty$ as $a$ is wandering and $\phi$ is non-isotrivial (see also \cite[Corollary 1.8]{Baker}), we have that 
\begin{equation}
\label{eq:Case5}
\lambda_v(\phi^n(\alpha),\gamma)\leq h(\phi^{n-m}(a))+ C_4(v,\phi,m) 
\end{equation} 
for all $n>m$ sufficiently large.  
\\[7pt] 
In summation, if we let $C_8(v,\phi,m)$ be the max of $C_i(v,\phi,m)$ for $i=1,2,4,5,7$ then \eqref{eq:Case1}, \eqref{eq:Case2}, \eqref{eq:Case3}, \eqref{eq:Case4} and  \eqref{eq:Case5} together imply that     
\begin{equation}\label{summary}
\lambda_v(\phi^n(a),\gamma)\leq \Big(\frac{d^m}{2}+2\Big)\cdot h(\phi^{n-m}(\alpha))+ C_8(v,\phi,m)
\end{equation} 
for all $n>m$, independent of what case the points $\alpha_{m,n}\in\phi^{-m}(\gamma)\subseteq\mathbb{P}^1(\overline{K})$ fall into. Note that we may also assume that $h(\phi^{n-m}(a))\geq1$ since $\alpha$ is wandering. In fact, $h(\phi^n(\alpha))\rightarrow+\infty$ by \cite[Corollary 1.8]{Baker}.  

On the other hand, the canonical height $\hat{h}_\phi:\mathbb{P}^1(\overline{K})\rightarrow \mathbb{R}_{\geq0}$ given by $\hat{h}_\phi=\displaystyle{\lim_{n\rightarrow\infty}\frac{ h(\phi^n(P))}{d^n}}$ satisfies two important properties: \vspace{.15cm} 
\[
\text{(A).\; $|h_\phi-h|\leq C_\phi$}\;\;\;\;\;\; \text{and}\;\;\;\;\;\; \text{(B).\; $\hat{h}_\phi(\phi^n(P))=d^n\hat{h}_\phi(P)$} \vspace{.15cm}
\]
for some constant $C_\phi$; see \cite[\S3.4]{SilvDyn}. In particular, it is straightforward to check that 
\begin{equation}\label{eq:limit8}
h(\phi^{n-m}(a))\leq\frac{1}{d^m}h(\phi^n(a))+(1+d^{-m})C_\phi.  
\end{equation} 
Hence, \eqref{summary} and \eqref{eq:limit8} together imply that \vspace{.1cm}   
\begin{equation*}
\lambda_v(\phi^n(a),\gamma)\leq\frac{1}{2}h(\phi^n(a))+\frac{2}{d^m}h(\phi^n(a))+C_9(v,\phi,m)
\end{equation*}
for some constant $C_9(v,\phi,m)$. In particular, dividing both sides of the inequality above by $h(\phi^n(a))$ we see that   
\begin{equation}\label{eq:limit10} 
\limsup_{n\rightarrow\infty}\frac{\lambda_v(\phi^n(a),\gamma)}{h(\phi^n(a))}\leq \frac{1}{2}+\frac{2}{d^m}. \vspace{.1cm} 
\end{equation}        
Here we use again that $h(\phi^n(\alpha))$ tends to infinity. On the other hand, the bound in \eqref{eq:limit10} is true for all $m\gg0$ (recall that $m$ was chosen so that Theorem \ref{thm:crossratios} holds). Therefore, $\limsup \lambda_v(\phi^n(a),\gamma)/h(\phi^n(a))\leq 1/2$ as claimed. 
\end{proof} 
\medskip 
As an application on Theorem \ref{thm:limit}, we prove the weak form of Silverman's Limit Theorem in characteristic $p$ from the Introduction; see Theorem \ref{thm:basic}. In fact, we establish the following stronger statement:   
\medskip
\begin{corollary}\label{cor:a_nandb_n} Suppose that $\phi\in \overline{\mathbb{F}}_p(t)(z)$ is non-isotrivial. Moreover, given a wandering point $a\in\mathbb{P}^1(\overline{\mathbb{F}}_p(t))$, write $\phi^n(a)=a_n/b_n$ for some coprime $a_n,b_n\in\overline{\mathbb{F}}_p[t]$. Then the following statements hold. \vspace{.1cm}  
\begin{enumerate}
\item[\textup{(1)}] If $\infty$ is not postcritical for $\phi$, then \[\limsup_{n\rightarrow\infty} \frac{\deg(a_n)}{\deg(b_n)}\leq2.\vspace{.2cm}\]    
\item[\textup{(2)}] If both $0$ and $\infty$ are not postcritical for $\phi$, then 
\[\frac{1}{2}\leq \liminf_{n\rightarrow\infty} \frac{\deg(a_n)}{\deg(b_n)} \leq\limsup_{n\rightarrow\infty} \frac{\deg(a_n)}{\deg(b_n)}\leq2. \vspace{.2cm} \]
\end{enumerate}
\end{corollary} 

\begin{proof}
If $\phi^n(a)=a_n(t)/b_n(t)$ for some coprime $a_n(t),b_n(t)\in \mathbb{F}_q[t]$, then \vspace{.2cm}  
\begin{equation}{\label{eg:infty}} 
\begin{split} 
\lambda_v(\phi^n(a),\infty)&=-\log\bigg(\frac{|b_n|_v}{\max\big\{|a_n|_v,|b_n|_v\big\}}\bigg)=\log\max\big\{|a_n|_v,|b_n|_v\big\}-\log(|b_n|_v)\\[8pt]
&=\max\{\deg(a_n),\deg(b_n)\}-\deg(b_n).\\[5pt] 
\end{split} 
\end{equation} 
On the other hand, $h(\phi^n(a))=\max\{\deg(a_n),\deg(b_n)\}$. Therefore, if $\infty$ is not post-critical, then Theorem \ref{thm:limit} and (\ref{eg:infty}) imply that  \vspace{.1cm} 
 \begin{equation}{\label{eg2:infty}}
\limsup_{n\rightarrow\infty}\,\frac{\max\{\deg(a_n),\deg(b_n)\}-\deg(b_n)}{\max\{\deg(a_n),\deg(b_n)\}}\leq\frac{1}{2}.\\[5pt]
\end{equation} 
Now let $\epsilon>0$ and suppose that $n$ is such that $\frac{\deg(a_n)}{\deg(b_n)}>2+\epsilon$. We will show that $n$ is bounded, and since $\limsup \frac{\deg(a_n)}{\deg(b_n)}$ is the smallest real number with this property, it follows that   
\[\limsup_{n\rightarrow\infty}\,\frac{\deg(a_n)}{\deg(b_n)}\leq 2. \vspace{.1cm} \]
For such $n$, we see that $\max\{\deg(a_n),\deg(b_n)\}=\deg(a_n)$, and thus \vspace{.1cm}  
\[\frac{\max\{\deg(a_n),\deg(b_n)\}}{\deg(b_n)}>2+\epsilon. \vspace{.1cm}  \]
In particular, with a little algebra we see that \vspace{.25cm} 
\[
\Scale[1.2]{
\frac{\max\{\deg(a_n),\deg(b_n)\}-\deg(b_n)}{\max\{\deg(a_n),\deg(b_n)\}}=1-\frac{\deg(b_n)}{\max\{\deg(a_n),\deg(b_n)\}}>\frac{-1}{2+\epsilon}+1=\frac{1}{2}+\frac{\epsilon}{2(2+\epsilon)}.} 
\vspace{.25cm}\]
On the other hand, since the remainder $\frac{\epsilon}{2(2+\epsilon)}$ is positive, (\ref{eg2:infty}) implies that $n$ is absolutely bounded. Therefore, we obtain the desired upper bound in Corollary \ref{cor:a_nandb_n}. Likewise, if $0$ is not post-critical for $\phi$, then we compute that \\[3pt]  
\begin{equation}\label{eg1:zero}
\limsup_{n\rightarrow\infty}\,\frac{\max\{\deg(a_n),\deg(b_n)\}-\deg(a_n)}{\max\{\deg(a_n),\deg(b_n)\}}=\limsup_{n\rightarrow\infty}\,\frac{\lambda_v(\phi^n(\alpha),0)}{h(\phi^n(\alpha))}\leq\frac{1}{2}. \vspace{.2cm}
\end{equation}
Let $0<\epsilon<\frac{1}{2}$ and suppose that $n$ is such that $\frac{\deg(a_n)}{\deg(b_n)}<\frac{1}{2}-\epsilon$. We will show that $n$ is bounded, and since $\liminf \frac{\deg(a_n)}{\deg(b_n)}$ is the largest real number with this property, it follows that \vspace{.1cm} 
\[\frac{1}{2}\leq \liminf_{n\rightarrow\infty}\,\frac{\deg(a_n)}{\deg(b_n)} \vspace{.2cm}\]
as claimed. Clearly, $\max\{\deg(a_n),\deg(b_n)\}=\deg(b_n)$ for such $n$, so that \\[4pt] 
\[\frac{\deg(a_n)}{\max\{\deg(a_n),\deg(b_n)\}}<\frac{1}{2}-\epsilon.\]
In particular, we see that \\[4pt] 
\[\frac{\max\{\deg(a_n),\deg(b_n)\}-\deg(a_n)}{\max\{\deg(a_n),\deg(b_n)\}}=1-\frac{\deg(a_n)}{\max\{\deg(a_n),\deg(b_n)\}}\geq\frac{1}{2}+\epsilon \vspace{.25cm}\]   
Hence, (\ref{eg1:zero}) implies that $n$ is absolutely bounded as claimed. 
\end{proof}

\end{document}